\documentclass[reqno]{amsart}

\usepackage{graphicx} 
\usepackage[margin= 1.0 in]{geometry}
\usepackage{amsfonts}
\usepackage{amsthm}
\usepackage{amsmath}
\usepackage{mathrsfs}
\usepackage{amssymb}
\usepackage{xtab}
\usepackage{tikz-cd} 
\usepackage{ytableau}
\usepackage{xfrac}
\usepackage{stackrel}
\usepackage{color}
\usepackage{enumerate}
\usepackage{mathtools}
\usepackage{pgfplots}
\pgfplotsset{compat=1.18}
\usepackage{float} 
\usepackage{hyperref}
\usepackage{stmaryrd}
\hypersetup{
   colorlinks=true,
   citecolor=blue,
   linkcolor=violet
   }

\usepackage{cleveref}
\usepackage{xcolor}
\usepackage{todonotes}

\newtheorem{theorem}{Theorem}[section]
\newtheorem{corollary}[theorem]{Corollary}
\newtheorem{proposition}[theorem]{Proposition}
\newtheorem{lemma}[theorem]{Lemma}

\theoremstyle{definition}
\newtheorem{remark}[theorem]{Remark}

\newtheorem{definition}[theorem]{Definition}

\crefname{inequality}{inequality}{inequalities} 
\Crefname{inequality}{Inequality}{Inequalities}

\def\Coh{\ensuremath{\mathrm{Coh}}}
\def\Hom{\ensuremath{\mathrm{Hom}}}
\def\Db{\ensuremath{\mathrm{D}^b}}
\def\R{\ensuremath{\mathbb R}}

\newcommand{\ch}[1]{\operatorname{ch}_{{#1}}}
\newcommand{\rk}{\operatorname{rk}}

\makeatletter
\renewcommand\subsection{%
  \@startsection{subsection}{2}{\z@}%
    {3.0ex \@plus 1ex \@minus .2ex}
    {1.5ex \@plus .2ex}
    {\normalfont\bfseries}
}
\makeatother

\title{Stronger Bogomolov--Gieseker type inequality on quintic threefold}
\author{Chunkai Xu}
\date{Nov 2025}

\begin{document}
\begin{abstract}
   We establish a stronger Bogomolov--Gieseker type inequality for slope-semistable sheaves on the smooth quintic threefold. Our approach combines a refined restriction theorem for tilt-stable objects with explicit Clifford-type bounds for semistable bundles on plane quintic curves. As a consequence, we obtain an explicit piecewise linear inequality on the Chern characters of any slope-semistable sheaf improving upon the classical Bogomolov--Gieseker bound and implying Toda’s conjectural inequality.
    
    The method also yields a stronger Bogomolov--Gieseker type inequality on smooth quintic surfaces.
    
    These results provide new evidence toward the existence of a Bridgeland stability condition of Gepner type on the quintic threefold.
\end{abstract}

\address{Mathematics Institute \\
 University of Warwick \\
 CV4 7AL \\
 United Kingdom}

\email{Chunkai.Xu@warwick.ac.uk}


\maketitle

\section{Introduction}
The classical result of Bogomolov and Gieseker provides a fundamental inequality for stable sheaves on smooth projective varieties.
\begin{theorem}[Bogomolov--Gieseker Inequality \cite{bogomolov_holomorphic_1979,gieseker_theorem_1979}]
Let $X$ be a smooth projective complex variety and $H$ an ample divisor on $X$. For any torsion-free $H$-slope-semistable sheaf $E$ on $X$, we have
\begin{equation*}
    \Delta(E) \cdot H^{\dim(X) - 2} \geq 0,
\end{equation*}
where the discriminant is defined by
\begin{equation*}
    \Delta(E) \coloneqq \ch{1}(E)^2 - 2\ch{0}(E)\ch{2}(E).
\end{equation*}
\end{theorem}

The Bogomolov--Gieseker (BG) inequality plays an important role in the study of stability conditions and moduli spaces of sheaves. It is a natural question to seek refinements of the BG inequality for special varieties; see for instance \cite[Question 5.1]{BM23}. 

Motivated by Gepner type stability conditions on graded matrix factorizations, Toda proposed in \cite{Toda2017GepnerPoint} a stronger inequality on quintic threefolds, and proved it for $\rk(E) = 2$ case.
\begin{theorem}[{\cite[Conjecture $1.1$]{Toda2017GepnerPoint}}]\label{TodaConj0}
    Let $X\subset \mathbb{P}^4_{\mathbb{C}}$ be a smooth quintic threefold and $H: = c_1(\mathcal{O}_X(1))$. Then for any torsion-free $H$-slope-stable sheaf $E$ on $X$ with $\ch{1}(E)/\rk (E) = -H/2$, we have
    \begin{equation}\label{TodaConj1}
        \frac{H\cdot \Delta(E)}{\rk(E)^2} > 1.5139\cdots,
    \end{equation}
    where the right-hand side of \eqref{TodaConj1} is an irrational real number lying in $\mathbb{Q}(e^{2\pi \sqrt{-1}/5})$. Equivalently,
    \begin{equation}\label{TodaConj2}
    \frac{H\ch{2}(E)}{H^3\rk(E)} < -0.02639\cdots
    \end{equation}
\end{theorem}

In this paper, we establish a stronger Bogomolov--Gieseker type inequality on quintic threefolds, which in particular implies Toda's conjecture.

For a torsion-free coherent sheaf $E$ on an $n$-dimensional polarized smooth projective variety $(X,H)$, we denote 
\begin{align*}
    \mu_H(E)\coloneqq\frac{H^{n-1}\ch{1}(E)}{H^n\rk(E)};\quad\xi_H(E)\coloneqq\frac{H^{n-2}\ch{2}(E)}{H^n\rk(E)}.
\end{align*}

\begin{theorem}[See \Cref{MainTheorem} and \Cref{MainCorollary}]
    Let $X\subset \mathbb{P}^4_{\mathbb{C}}$ be a smooth quintic threefold and $H: = c_1(\mathcal{O}_X(1))$. Let $F$ be a torsion-free $H$-slope-semistable sheaf in $\Coh(X)$. Assume that
    \begin{equation*}
        \mu_H(F)\in [-1,1].
    \end{equation*}
    Then
    \begin{equation}\notag 
         \xi_H(F)\leq 
    \begin{cases}
    -\frac{1}{2}|\mu_H(F)| & \text{when } 0\leq |\mu_H(F)| \leq \frac{1}{4}; \\[0ex]
    \frac{1}{2}|\mu_H(F)| - \frac{1}{4} & \text{when } |\mu_H(F)|\in[\frac{1}{4},\frac{5}{13}]\cup[\frac{8}{13},\frac{3}{4}]; \\[0ex]
    -\frac{3}{20}|\mu_H(F)|  & \text{when } \frac{5}{13} \leq |\mu_H(F)|\leq \frac{6}{13}; \\[0ex]
    \frac{1}{2}|\mu_H(F)|-\frac{3}{10} & \text{when } \frac{6}{13}\leq |\mu_H(F)| \leq \frac{7}{13}; \\[0ex]
    \frac{23}{20}|\mu_H(F)| - \frac{13}{20} & \text{when } \frac{7}{13}\leq |\mu_H(F)| \leq \frac{8}{13}; \\[0ex]
    \frac{3}{2}|\mu_H(F)| - 1 & \text{when } \frac{3}{4}\leq |\mu_H(F)| \leq 1.
\end{cases}
    \end{equation}
\end{theorem}

Along the way, we also obtain a stronger Bogomolov--Gieseker type inequality for quintic surfaces.
\begin{theorem}[See \Cref{ChInequalityOnSurface} and \Cref{CorollaryOnSheaves}]
    Let $S_5\subset \mathbb{P}^3_{\mathbb{C}}$ be a smooth quintic surface and $H\coloneqq c_1(\mathcal{O}_X(1))$. Let $F$ be a torsion-free $H$-slope-semistable sheaf in $\Coh(X)$ with
    \begin{equation*}
        \mu_H(F)\in (0,1). 
    \end{equation*}
    Then the following piecewise inequality holds:
    \begin{equation}\notag
        \xi_H(F)\leq 
    \begin{cases}
    \frac{13}{20}\mu_H(F) - \frac{3}{20} & \text{when } 0 < \mu_H(F) \leq \frac{7}{47}; \\[0ex]
    -\frac{5}{14}\mu_H(F)  & \text{when } \frac{7}{47} \leq  \mu_H(F)\leq \frac{7}{20}; \\[0ex]
    \frac{1}{2}\mu_H(F)-\frac{3}{10} & \text{when } \frac{7}{20}\leq \mu_H(F) \leq \frac{13}{20}; \\[0ex]
    \frac{19}{14}\mu_H(F) - \frac{6}{7} & \text{when } \frac{13}{20}\leq \mu_H(F)< \frac{40}{47}; \\[0ex]
    \frac{7}{20}\mu_H(F) & \text{when } \frac{40}{47}\leq \mu_H(F)< 1.
\end{cases}
    \end{equation}
\end{theorem}

Our proof relies on the notion of tilt-stability. The strategy is to employ a generalized restriction theorem that allows us to reduce to the case of quintic surfaces, where the desired bounds follow from the Riemann--Roch theorem combined with Clifford-type inequalities for plane quintic curves.

This approach parallels the methods of \cite{li_stability_2019,koseki_stability_2022,liu_stability_2022}, which established analogous inequalities for some complete intersection type Calabi--Yau threefolds. The main improvement here lies in a refined restriction theorem (\Cref{RestrictionTheorem}), which weakens the hypothesis and allows restriction to $|H|$ instead of $|2H|$. As a consequence, we may directly use Clifford-type bounds for plane curves, leading to sharper estimates. The method employed here extends to weighted complete intersection hypersurfaces as well, though the computations in each case tend to be considerably more demanding.

Stronger Bogomolov--Gieseker type inequalities on Calabi--Yau threefolds are known to play a crucial role in constructing Bridgeland stability conditions, as demonstrated in \cite{li_stability_2018,li_stability_2019,koseki_stability_2022,liu_stability_2022} and the recent preprint \cite{feyzbakhsh2025stabilityconditionscalabiyauthreefolds}. In fact, Toda's conjecture (\Cref{TodaConj0}) was originally motivated by the search for a Gepner-type stability condition on the quintic threefold. However, as far as the author is aware, the existing inequalities are not yet sufficient for this purpose due to certain technical obstacles. 

The construction of such a Gepner-type stability condition on the quintic remains an intriguing direction for future work.

\section*{Acknowledgments}
The author would like to thank Shengxuan Liu for introducing him to the problem studied in this article. He is deeply grateful to Chunyi Li for his patient and insightful guidance throughout the project. He also thanks Naoki Koseki, Peize Liu and Zhiyu Liu for many helpful comments and suggestions that improved the exposition.

The author is supported by the Warwick Mathematics Institute Centre for Doctoral Training, and gratefully acknowledges funding from the University of Warwick. The author is also partially supported by the Royal Society URF\textbackslash R1\textbackslash 201129 “Stability condition and application in algebraic geometry”.

\section{Preliminary}
In this section, we recall the basic notions of slope stability and tilt-stability on smooth projective varieties, following conventions of \cite{bridgeland_stability_2007,bayer_bridgeland_2013}. 

Throughout, $X$ denotes a smooth projective variety of dimension $n$, where $n = 2$ or $3$. We denote by $\Db(X)\coloneqq \Db(\Coh(X))$ the bounded derived category of coherent sheaves on $X$.

\subsection{Stability condition: notations and conventions}

Let $B\in \mathrm{NS}(X)_{\mathbb{R}}$ be a real divisor class. We denote the \textbf{twisted Chern characters} of an object in $\Db(X)$ by 
\begin{align*}
    &\ch{0}^B(E) = \ch{0}(E) = \rk(E);   \\
    &\ch{1}^B(E) = \ch{1}(E) - B\ch{0}(E); \\
    &\ch{2}^B(E) = \ch{2}(E) -B\ch{1}(E) + \tfrac{1}{2}B^2\ch{0}(E);\\
    &\ch{3}^B(E) = \ch{3}(E)- B\ch{2}(E) + \tfrac{1}{2}B^2\ch{1}(E) - \tfrac{1}{6}B^3\ch{0}(E).
\end{align*}

\begin{definition}[\textbf{Slope stability}]
    Let $H$ be an ample divisor on $X$. The slope of a coherent sheaf $F$ with respect to $H$ is 
    \begin{equation*}
        \mu_H(F) \coloneqq \begin{cases}
            \frac{H^{n - 1} \ch{1}(F)}{H^n\ch{0}(F)}, & \text{if } \ch{0}(F)\neq 0, \\[1ex]
            +\infty, & \text{if } \ch{0}(F) = 0.
        \end{cases}
    \end{equation*}
    A coherent sheaf $F$ is called \textbf{$H$-slope-(semi)stable} if for every nontrivial subsehaf $E\subset F$,
    \begin{equation*}
        \mu_H(E) < (\leq ) \mu_H(F/E)
    \end{equation*}
\end{definition}

\begin{proposition}[\textbf{Harder--Narasimhan filtration}] For any coherent sheaf $F$ on $X$, there exists a unique filtration
\begin{equation*}
    0 = F_0 \subset F_1 \subset \cdots\subset F_n = F,
\end{equation*}
such that each quotient $F_i/F_{i - 1}$ is $H$-slope-semistable and 
\begin{equation*}
    \mu_H(F_1/F_0) > \mu_H(F_2/F_1)>\cdots > \mu_H(F_n/F_{n-1}).
\end{equation*}
We write
\begin{equation*}
    \mu^+_H(F)\coloneqq\mu_H(F_1/F_0), \qquad\mu^-_H(F) \coloneqq\mu_H(F_n/F_{n - 1})
\end{equation*}
for the maximal and minimal slope of the factors.
\end{proposition}

For any real number $\beta\in\mathbb{R}$, there exists a \textit{torsion pair} $(\mathcal{T}_{\beta,H}, \mathcal{F}_{\beta,H})$ in $\Coh(X)$ as follows:
\begin{equation*}
    \mathcal{T}_{\beta,H} = \{E\in\Coh(X)\mid \mu^-_H(E) > \beta \}, \qquad \mathcal{F}_{\beta,H} = \{E\in\Coh(X)\mid \mu^+_H(E) \leq \beta \}.
\end{equation*}

\begin{definition}[\textbf{Tilted heart}]
    We define the tilted heart
    \begin{equation*}
        \Coh^{\beta, H}(X) : = \langle\mathcal{T}_{\beta,H}, \mathcal{F}_{\beta,H}[1]\rangle.
    \end{equation*}
    to be the extension-closure of the above torsion pair in $\Db(X)$.
\end{definition}
By standard tilting theory in \cite{happel_tilting_1996}, $\Coh^{\beta,H}(X)$ forms the heart of a bounded $t$-structure in $\Db(X)$, hence an abelian category. 

\begin{definition}[\textbf{Tilt slope and tilt-stability}] For $E\in \Coh^{\beta,H}(X)$, $\alpha\in\mathbb{R}$, the  $\nu_{\alpha,\beta,H}$-\textbf{tilt slope} of $E$ is
\begin{equation*}
    \nu_{\alpha,\beta,H}(E) \coloneqq \begin{cases}
        \frac{H^{n-2}\ch{2}(E) - \alpha H^n\ch{0}(E)}{H^{n- 1} \ch{1}^{\beta H}(E)} &\text{if } H^{n-1}\ch{1}^{\beta H} \neq 0, \\[1ex]
        +\infty &\text{if } H^{n-1}\ch{1}^{\beta H} = 0.
    \end{cases}
\end{equation*}
An object $E\in\Coh^{\beta,H}(X)$ is called $\nu_{\alpha,\beta,H}$-tilt-(semi)stable if for any nontrivial subobject $F\subset E$ in $\Coh^{\beta,H}(X)$,
\begin{equation*}
    \nu_{\alpha,\beta,H}(F) < (\leq ) \nu_{\alpha,\beta,H}(E/F).
\end{equation*}
An object $E\in \Db(X)$ is called $\nu_{\alpha,\beta,H}$-tilt-(semi)stable if $E[n]\in\Coh^{\beta,H}(X)$ is $\nu_{\alpha,\beta,H}$-tilt-(semi)stable for some integer $n$.
\end{definition}

As in slope stability, when $\alpha > \frac{1}{2}\beta^2$, the $\nu_{\alpha,\beta,H}$-stability admits the Harder--Narasimhan filtration property. For an object $E\in \Coh^{\beta,H}(X)$, we denote by $\nu^+_{\alpha,\beta,H}(E)$ and $\nu^-_{\alpha,\beta,H}(E)$ the maximum and minimum slopes of its Harder--Narasimhan filtration factors.

\begin{definition}[\textbf{$H$-discriminant}]
    For $E\in \Db(X)$, define
    \begin{equation*}
        \overline{\Delta}_H(E) : = (H^{n-1}\ch{1}(E))^2 - 2H^n\ch{0}(E)\cdot H^{n-2}\ch{2}(E).
    \end{equation*}
\end{definition}

The classical BG inequality extends to tilt-semistable objects.
\begin{theorem}[{\cite[Theorem 7.3.1]{bayer_bridgeland_2013}}, {\cite[Proposition 2.21]{piyaratne_moduli_2019}}]
    Let $X$ be a smooth projective variety, and $H\in \mathrm{NS}(X)_{\mathbb{R}}$ an ample class. If $E$ is $\nu_{\alpha,\beta,H}$-tilt-semistable for some $\alpha > \frac{1}{2}\beta^2$, then $\overline{\Delta}_H(E) \geq 0$.
\end{theorem}

\subsection{Useful Lemmas}
Let $X$ be a smooth projective variety of dimension $n$ and $H\in \mathrm{NS}(X)_{\mathbb{R}}$ be a real ample divisor class. For $E\in \Db(X)$, define
\begin{equation*}
    \overline{v}_H(E) \coloneqq (H^n\ch{0}(E),H^{n-1}\ch{1}(E),H^{n-2}\ch{2}(E)),
\end{equation*}
and, when  $H^n\ch{0}(E) \neq 0$,
\begin{equation*}
    \text{and } p_H(E) \coloneqq \left(\frac{H^{n-1}\ch{1}(E)}{H^n\ch{0}(E)}, \frac{H^{n-2}\ch{2}(E)}{H^n\ch{0}(E)}\right) = (\mu_H(E),\xi_H(E)). 
\end{equation*}
\begin{remark}
    For a smooth hypersurface $Y\in |mH|$ for $m > 0$, and for $0\leq i\leq n - 1$, we have
    \begin{equation*}
        mH^{n-i}\ch{i}(F) = H^{n-i-1}_Y\ch{i}(F|_Y).
    \end{equation*}
    Hence, $\mu_H$ and $\xi_H$ are invariant under restriction to hypersurfaces. In particular, we have 
    \begin{equation*}
        p_H(F) = p_{H_Y} (F|_Y).
    \end{equation*}
    
\end{remark}

Throughout, we fix real parameters $\alpha,\beta\in\mathbb{R}$ with $\alpha > \frac{1}{2}\beta^2$. In the $(x,y)$-plane where $x = \frac{H^{n-1}\ch{1}(\cdot)}{H^n\ch{0}(\cdot)}$ and $y = \frac{H^{n-2}\ch{2}(\cdot)}{H^n\ch{0}(\cdot)}$, every point $(\beta,\alpha)$ with $\alpha\geq \frac{1}{2}\beta^2$ corresponds to a tilt-stability condition. By BG inequality, every stable character $p_H(E)$ lies on or under the parabola $y = \frac{1}{2}x^2$.

\begin{lemma}\label{StructureOfWalls}
    Let $E\in \Coh^{\beta_0,H}(X)$ be $\nu_{\alpha_0,\beta_0,H}$-tilt-stable  for some $\alpha_0>\frac{1}{2}\beta^2$, then:
    \begin{enumerate}[(a)]
        \item (Openness) There exists an open set of neighborhood $U$ of $(\beta_0,\alpha_0)$ such that for any $(\beta,\alpha)\in U$, the object $E$ is $\nu_{\alpha,\beta,H}$-tilt-stable.
        \item (Bertram's nested wall theorem) $E$ is $\nu_{\alpha,\beta,H}$-tilt-stable for all points $(\beta,\alpha)$ on the line through  $(\beta_0,\alpha_0)$ and $p_H(E)$ satisfying $\alpha > \frac{1}{2}\beta^2$. The statement also holds for semistable case.
        \item[(b$'$)] Let $F$ be an object in $\Coh^{\beta_0,H}(X)$ such that $p_H(F)$ is on the line through the points $(\beta_0,\alpha_0)$ and $p_H(E)$, then $\nu_{\alpha_0,\beta_0,H}(E) = \nu_{\alpha_0,\beta_0,H}(F)$.
        \item(Destabilizing walls) The set $\{(\beta,\alpha)\in\mathbb{R}^2 | \alpha > \frac{1}{2}\beta^2, E \text{ is strictly } \nu_{\alpha,\beta,H}\text{-tilt-semistable} \}$ is empty or a union of line segments and rays.
    \end{enumerate}
\end{lemma}
\begin{proof}
    See \cite[Corollary 3.3.3]{bayer_bridgeland_2013}, \cite[Appendix B]{bayer_space_2016}, also \cite[Lemma 2.9]{li_stability_2019}.
\end{proof}

The following lemma from \cite{BBFHMRS24} will be crucial in the technique of deforming tilt-stability. 
\begin{lemma}[{\cite[Proposition 4.8]{BBFHMRS24}}]\label{LemmaDiscriminant}
    Let $E$ be a strictly $\nu_{\alpha,\beta,H}$-tilt-semistable object with $\nu_{\alpha,\beta,H}(E)\neq +\infty$. Then for any Jordan--H\"older factor $E_i$ of $E$, we have
    \begin{equation*}
        \overline{\Delta}_H(E_i) \leq \overline{\Delta}_H(E).
    \end{equation*}
    The equality holds only when $\overline{v}_H(E_i)$ is proportional to $\overline{v}_H(E)$ and $\overline{\Delta}_H(E) = \overline{\Delta}_H(E_i) = 0$.
\end{lemma}

\begin{remark}\label{RemarkOnDiscriminats}
    If $E\in \Db(X)$ with $\mu_H(E) \in (0,1)$ is $\nu_{\alpha,0,H}$-tilt-stable and strictly $\nu_{\alpha_0,0,H}$-tilt-semistable for some $\alpha,\alpha_0 > 0$, then $\overline{\Delta}_H(E) > 0$. 

    Indeed, if $\overline{\Delta}_H(E) = 0$, all Jordan--H\"older factors of $E$ at $(0,\alpha_0)$ share the same $p_H$ as $E$, then these objects will destabilize $E$ everywhere, which contradicts our assumption.
\end{remark}

We will use the following relation between slope stability and tilt-stability:
\begin{lemma}[{\cite[Lemma 2.7]{bayer_space_2016}}]\label{LargeVolumeLimit}
    If $E\in\Coh(X)$ is an $H$-slope-stable torsion-free sheaf and $H^{n-1}\ch{1}^\beta(E) > 0$, then $E\in \Coh^{\beta,H}(X)$ and it is $\nu_{\alpha,\beta,H}$-tilt-stable for $\alpha\gg 0$.
\end{lemma}

\begin{lemma}[Riemann--Roch for quintic surface]\label{RRQuinticSurface}
    For a quintic surface $S$ and $E\in \Db(S)$, we have
    \begin{equation*}
        \chi(E) = \ch{2}(E) - \tfrac{1}{2}H\ch{1}(E)+ H^2\ch{0}(E).
    \end{equation*}
\end{lemma}

\section{Clifford-type bound on plane quintic curves}\label{section3}
In this section we recall some Clifford-type inequalities for vector bundles on smooth plane quintic curves, which will later serve as a key input in proving the stronger Bogomolov--Gieseker type inequalities for quintic surfaces and threefolds.

We begin with the following result from \cite{feyzbakhsh_higher_2021}, which provides an explicit bound for the number of global sections of a semistable bundle on a plane curve of degree $l\geq 5$. We denote by $h^0(E) := \dim H^0(C,E)$.

\begin{proposition}[{\cite[Theorem $5.5$]{feyzbakhsh_higher_2021}}]
    Let $C\subset \mathbb{P}^2$ be a smooth irreducible plane curve of degree $l\geq 5$. Let $E$ be a semistable vector bundle of rank $r$ and  degree $d$ on $C$ satisfying $0\leq d \leq \frac{rl(l-3)}{2}$. Then
    \begin{equation}\notag
        h^0(E) \leq 
        \begin{cases}
            r + \left(\frac{3}{2l}+ \frac{d}{2rl^2}\right) d, & \text{if } d\geq rl, \\[0ex]
            \max \{3r + d - rl, r + \frac{rl+r}{rl^2 - d}d \}, & \text{if } d < rl.
        \end{cases}
    \end{equation}
\end{proposition}

For our purposes we only require a simplified version for smooth plane quintic curves $C_5\subset \mathbb{P}^2$ (of genus $6$).
By combining Serre duality and the Riemann--Roch theorem, one obtains an explicit bound for the entire slope range $[0,10]$. Denote $H_C:= \mathcal{O}_{\mathbb{P}^2}(1)|_{C}$, for convenience, we will use $\mu_{H_C}$ in the sequel instead of $\frac{d}{r}$. Note that $\mu_{H_C}(E) = \frac{d}{5r}$ in this case. 

\begin{proposition}\label{CliffordType}
    Let $C\subset\mathbb{P}^2$ be a smooth irreducible plane quintic curve, and let $E$ be a semistable vector bundle of rank $r$ and degree $d$ on $C$. Then
    \begin{equation}
        \frac{h^0(E)}{r} \leq 
        \begin{cases}
            \frac{7}{2}\mu_{H_C}(E) - 1 & \text{if } \frac{8}{7}\leq\mu_{H_C}(E)\leq 2; \\[0ex]
            3 & \text{if } 1\leq \mu_{H_C}(E) \leq \frac{8}{7}; \\[0ex]
            5\mu_{H_C}(E) -2   & \text{if } \frac{6}{7}\leq \mu_{H_C}(E) < 1; \\[0ex]
            1 + \frac{3}{2}\mu_{H_C}(E) & \text{if } 0 < \mu_{H_C}(E) \leq \frac{6}{7}. 
        \end{cases}
    \end{equation}
    Moreover, if $\mu_{H_C}(E)> 2$, then $\frac{h^0(E)}{r} = 5\mu_{H_C}(E) - 5$.
\end{proposition}

When $E$ is not semistable, analogous estimates can be obtained in terms of the slopes of the Harder--Narasimhan factors.

For a real-value function $h:I\to \R$, where $I\subset \R$ and an interval $[a,b]\subset I$, we denote $\overline{h}_{a,b}:[a,b]\to \R$ as the `convex roof' of $h|_{[a,b]}$. More precisely, $\overline{h_{a,b}}$ is the smallest function that is convex from below and satisfies $\overline{h}_{a,b}\geq h|_{[a,b]}$.

\begin{lemma}\label{lem:convexfunction}
    Let $(X,H)$ be a polarized smooth variety such that every $H$-slope-semistable torsion-free sheaf $E$ satisfies 
    \begin{align*}
         \frac{h^0(E)}{\rk(E)}\leq h(\mu_H(E))
    \end{align*}
    for some function $h$. Then for every torsion-free sheaf $E$, we have 
    \begin{align*}
        \frac{h^0(E)}{\rk(E)} \leq \overline{h}_{\mu^-_{H}(E),\mu^+_H(E)}(\mu_H(E)).
    \end{align*}
\end{lemma}
\begin{proof}
    Let
    \begin{equation*}
        0=E_0\hookrightarrow E_1\hookrightarrow\cdots\hookrightarrow E_n = E,
    \end{equation*}
    be the Harder--Narasimhan filtration of $E$.
    For each quotient $E_i/E_{i-1}$, semi-stability and our assumption yield
    \begin{equation*}
        \frac{h^0(E_i/E_{i-1})}{\rk(E_i/E_{i-1})} \leq h(\mu_H(E_i/E_{i-1})). 
    \end{equation*}
    Using these inequalities and additivity of rank and degree, we have
    \begin{align*}
        \frac{h^0(E)}{\rk(E)}&\leq \sum\limits_{i=1}^n\frac{1}{\rk(E)}h^0(E_i/E_{i-1}) = \sum\limits_{i=1}^n\frac{\rk(E_i/E_{i-1})}{\rk(E)}\frac{h^0(E_i/E_{i-1})}{\rk(E_i/E_{i-1})} \\
        &\leq \sum\limits_{i=1}^n\frac{\rk(E_i/E_{i-1})}{\rk(E)}h(\mu_H(E_i/E_{i-1})) \\
        &\leq \sum\limits_{i=1}^n\frac{\rk(E_i/E_{i-1})}{\rk(E)}\overline{h}_{\mu^-_{H}(E),\mu^+_H(E)}(\mu_H(E_i/E_{i-1})) \\
        &\leq \overline{h}_{\mu^-_{H}(E),\mu^+_H(E)}(\mu_H(E)),
    \end{align*}
    where the last line follows from convexity of $\overline{h}_{\mu^-_{H}(E),\mu^+_H(E)}$.
\end{proof}


\begin{figure}[H]
    \centering
    \begin{tikzpicture}[line cap=round,line join=round,x=10cm,y=10cm]
  \begin{axis}[
    axis lines=middle,
    xmin=0, xmax=3,
    ymin=0, ymax=8,
    xlabel={$\mu_H$},
    ylabel = {$\frac{h^0}{r}$},
    ylabel style={font=\Large},
    xtick={6/7,1,8/7}, 
    xticklabels={$\frac{6}{7}$, $1$, $\frac{8}{7}$},
    ytick={1,3},
    yticklabels={$1$,$3$},
    width= 14cm, height = 10cm
  ]
    \addplot[thick,domain=0:6/7]{1+ 1.5*x};
    \addplot[thick,domain=6/7:1]{5*x - 2};
    \addplot[thick,domain=1:8/7]{3};
    \addplot[thick,domain=8/7:2]{3.5*x - 1};
    \addplot[thick,domain=2:3]{5*x - 5};

    \draw[thick,fill=white] (axis cs:2,5) circle (1.5pt);
    \filldraw (axis cs:2,6) circle (1.5pt);

    \draw [line width=1pt,dash pattern=on 1pt off 5pt] (1,0)-- (1,3);
    \draw [line width=1pt,dash pattern=on 1pt off 5pt] (6/7,0)-- (6/7,16/7);
    \draw [line width=1pt,dash pattern=on 1pt off 5pt] (8/7,0)-- (8/7,3);
    \draw [line width=1pt,dash pattern=on 1pt off 5pt] (1,3)-- (0,3);
    
  \end{axis}
\end{tikzpicture}

    \caption{Clifford-type bound on plane quintic curves}
    \label{fig:H0boundQuinticCurves}
\end{figure}
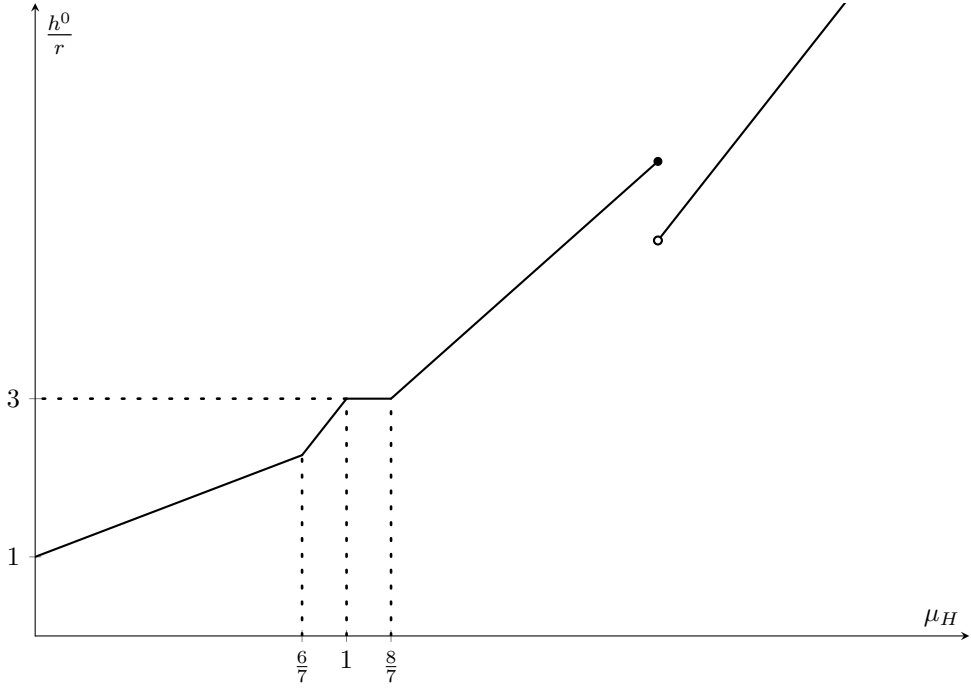

With $h$ defined as 
\begin{equation}
        h(x) = 
        \begin{cases}
            5x - 5 & \text{if } 2< x; \\
            \frac{7}{2}x - 1 & \text{if } \frac{8}{7}\leq x\leq 2; \\
            3 & \text{if } 1\leq x \leq \frac{8}{7}; \\ 
            5x -2   & \text{if } \frac{6}{7}\leq x < 1; \\ 
            1 + \frac{3}{2}x & \text{if } 0 \leq x \leq \frac{6}{7}.; \\ 
            0 & \text{if } x < 0.
        \end{cases}
\end{equation}
We compute the following two special cases.

\begin{corollary}\label{H0Bound}
     Let $C\subset \mathbb{P}^2$ be a smooth plane quintic curve, and let $E$ be a vector bundle of rank $r$ and degree $d$ on $C$. Assume that $0 \leq \mu^-_{H_C}(E) < \mu^+_{H_C}(E) \leq \frac{6}{7}$, then \begin{equation}
        \frac{h^0(E)}{r} \leq 1 + \frac{3}{2} \mu_{H_C}(E).
    \end{equation}
    Equivalently, $\overline{h}_{0,\frac{6}{7}}(x) = 1+ \frac{3}{2}x$.
\end{corollary}

\begin{corollary}\label{H0Bound1}
    Let $C\subset \mathbb{P}^2$ be a smooth plane quintic curve, and let $E$ be a vector bundle of rank $r$ and degree $d$ on $C$. Assume that $\mu^+_{H_C}(E) \leq 2$, then \begin{equation}
        \frac{h^0(E)}{r} \leq h(\mu^+_{H_C}(E)).
    \end{equation}
    Equivalently, $\overline{h}_{-\infty,\mu^+_{H_C}(E)}(x) = h(\mu^+_{H_C}(E))$. 
\end{corollary}

\begin{remark}
    \Cref{H0Bound} and \Cref{H0Bound1} provide the precise estimates for the ranges of slopes that appear in the restriction arguments for quintic surfaces (\Cref{ChInequalityOnSurface}).

    In particular, \Cref{H0Bound} will control the case of small slope, while \Cref{H0Bound1} covers the complementary range up to $\mu^+_{H_C}(E) \leq 2$.
\end{remark}

\section{Stronger Bogomolov--Gieseker type inequality on quintic surface and quintic threefolds}
In this section, we prove our main theorem. The proof proceeds by reducing the problem to a quintic surface via a restriction theorem, then applying a Clifford-type inequality on curves to derive the bound. The proof here is similar to \cite[Proposition 5.2]{li_stability_2019}, see also \cite[Section 3]{Naoki23_BG_hypersurfaces}. 

\subsection{Lemmas: restriction theorem and reduction lemma}
To begin with, we have the following restriction theorem, which generalizes \cite{feyzbakhshEffectiveRestrictionTheorem2022}, \cite[Lemma 5.1]{li_stability_2019}. For unordered endpoints $a,b \in \mathbb{R}$, we denote
\[
\llbracket a,b \rrbracket := [\min(a,b),\, \max(a,b)].
\]
\begin{lemma}[\textbf{Restriction theorem}]\label{RestrictionTheorem}
    Let $(X,H)$ be a polarized smooth projective variety with dimension $n \geq 2$. Let $E\in \Coh(X)$. Suppose there exists $\alpha > 0$ and $m\in\mathbb{Z}_{> 0 }$ such that both $E$ and  $E(-mH)[1]$ are in $\Coh^{0,H}(X)$ and $\nu_{\alpha,0,H}$-tilt-semistable.

    Then for a smooth irreducible subvariety $Y\in |mH|$, the restricted sheaf $E|_Y$ has $\rk(E) = \rk(E|_{Y})$, $H_Y^{n-2}\ch{1}(E|_Y) = mH^{n-1}\ch{1}(E)$ and when $n \geq 3$, $\ch{2}(E|_Y) = mH\ch{2}(E)$. Moreover, we have
    \begin{align*}
    [\mu_{H_Y}^-(E|_Y),\mu^+_{H_Y}(E|_Y)]\subseteq \llbracket \tfrac{m}{2}+\nu_{\alpha,0,H}(E), \tfrac{m}{2}+\nu_{\alpha,0,H}(E(-mH)[1]) \rrbracket.
    \end{align*} In particular, if $\nu_{\alpha,0,H}(E) = \nu_{\alpha,0,H}(E(-mH)[1])$, then $E|_Y$ is $H_Y$-slope-semistable.
\end{lemma}
\begin{proof}
    Since $E(-mH)[1]$ is $\nu_{\alpha,0,H}$-tilt semistable, for any torsion sheaf $T$ supported on a variety with codimension not less than $2$, we have $\text{Hom}(T,E(-mH)[1]) = 0$. In particular, $E$ is a reflexive sheaf. 
    For a smooth irreducible $Y\in |mH|$, the restricted sheaf $E|_Y$ is a pure sheaf on $Y$. In addition, $\rk(E) = \rk (E|_Y)$, $H^{n-2}_Y \ch{1}(E|_Y) = mH^{n-1}\ch{1}(E)$.

    Denote the embedding by $\iota:Y\hookrightarrow X$. Suppose there is a subobject $F\hookrightarrow E|_Y$ in $\Coh(Y)$ such that $\mu_{H_Y}(F) > \tfrac{m}{2} + \max\{\nu_{\alpha,0,H}(E), \nu_{\alpha,0,H}(E(-mH)[1]) \}$, then
    \begin{align*}
         \nu_{\alpha,0,H}(\iota_*(F)) &= \frac{H^{n-2}\ch{2}(\iota_*(F))}{H^{n-1}\ch{1}(\iota_*(F))} = \frac{H^{n-2}_Y\ch{1}(F) - \frac{1}{2}mH^{n-1}_Y\rk(F)}{H^{n-1}_Y \rk(F)} \\
        &= \mu_{H_Y}(F) - \tfrac{m}{2} > \max\{\nu_{\alpha,0,H}(E), \nu_{\alpha,0,H}(E(-mH)[1]) \}.
    \end{align*}
    However, the object $\iota(E|_Y)$ is the extension of $E$ by $E(-mH)[1]$ in $\Coh^{0,H}(X)$, and any subobject of $\iota_*(E|_Y)$ has $\nu_{\alpha,0,H}$-slope less than or equal to $\max\{\nu_{\alpha,0,H}(E), \nu_{\alpha,0,H}(E(-mH)[1]) \}$, which gives the contradiction. In the same way, we may argue for any quotient $E|_Y\rightarrow G$, we have 
    \begin{equation*}
        \mu_{H_Y}(G)\geq \tfrac{m}{2}+ \min\{\nu_{\alpha,0,H}(E), \nu_{\alpha,0,H}(E(-mH)[1]) \}.
    \end{equation*}
    Therefore, 
    \begin{align*}
        [\mu_{H_Y}^-(E|_Y),\mu^+_{H_Y}(E|_Y)]\subseteq \llbracket \tfrac{m}{2}+\nu_{\alpha,0,H}(E), \tfrac{m}{2}+\nu_{\alpha,0,H}(E(-mH)[1]) \rrbracket.
    \end{align*}
\end{proof}

\begin{remark}
    Compared to \cite{feyzbakhshEffectiveRestrictionTheorem2022} and  \cite[Lemma 5.1]{li_stability_2019}, the above lemma drops the assumption on the tilt-slope of $E$ and $E(-mH)$. While the restricted sheaf is not necessarily semistable, the slope of Harder--Narasimhan factors can be controlled. In practice, this allows us to restrict to $|H|$, instead of $|2H|$.
\end{remark}

The following formulae will be useful
\begin{align}
   & \nu_{0,0,H}(F)  = \frac{\xi_{H}(F)}{\mu_H(F)}; & 
    \nu_{0,0,H}(F^\vee) = - \frac{\xi_{H}(F)}{\mu_H(F)};\notag \\
  &  \nu_{0,0,H}(F(-H)) + \tfrac{1}{2} = \frac{\xi_H(F)- \frac{1}{2}\mu_H(F)}{\mu_H(F) - 1}; 
   & \nu_{0,0,H}(F^\vee(H)) + \tfrac{1}{2} = \frac{\xi_H(F)- \frac{1}{2}}{1-\mu_H(F)} + \tfrac{3}{2}. \label{eqnu}
\end{align} 
Also, when $\xi_H(F) \leq \mu_H^2(F) - \frac{1}{2}\mu_H(F)$, we have $\nu_{0,0,H}(F)\leq \nu_{0,0,H}(F(-H))$.

As in \cite{Naoki23_BG_hypersurfaces}, we use the following terminology:
\begin{definition}\label{dfnStarShape}
    Fix an integer $d\geq 0$. We say that a function $f:(0,1)\rightarrow\mathbb{R}$ is \textit{star-shaped along the line} $\beta = d$ if the following condition holds: for every real number $t\in(0,1)$, the line segment connecting the point $(t,f(t))$ and the point $(d,d^2/2)$ is above the graph of $f$.
\end{definition}

We will use the following variant of \cite[Lemma 3.3]{Naoki23_BG_hypersurfaces}, which allows us to assume $F$ is everywhere stable:
\begin{lemma}\label{LemmaReductionToStableObjects}
    Let $d\geq 1$ be an integer, and $f:(0,1)\rightarrow \mathbb{R}$ be a star-shaped function along the lines $\beta = 0, d$, satisfying
    \begin{equation*}
         f(t) \leq \tfrac{1}{2}t^2
    \end{equation*}
    for every $t\in (0,1)$. Assume that there exists objects $F'\in \Db(X)$ satisfying the following conditions:
    \begin{enumerate}[{\ \ }(a)]
        \item $F'$ is either $\nu_{0,\alpha}$-tilt-semistable for some $\alpha > 0$, or $\nu_{d,\alpha'}$-tilt-semistable for some $\alpha' > d^2/2$.
        \item  $\mu_H(F')\in (0,1)$, $\xi_H(F') > f(\mu_H(F'))$.
    \end{enumerate}
    Then we can choose such an object $F$ so that:
    \begin{enumerate}
        \item  $F$ is a reflexive coherent sheaf.
        \item $F\in\Coh^{0,H}(X)$ is $\nu_{\alpha,0,H}$-tilt-stable for any $\alpha > 0$.
        \item $F[1] \in \Coh^{d,H}(X)$ is $\nu_{\alpha',d,H}$-tilt-stable for any $\alpha' > d^2/2$.
    \end{enumerate}
\end{lemma}
\begin{proof}
    Take such an $F'$. 
    Since the value of discriminant is a nonnegative integer, we may assume that $F'$ is an object with the minimal $\overline{\Delta}_H$ of all such objects. Suppose $F'$ becomes strictly $\nu_{\alpha,0,H}$-tilt-semistable for some $\alpha > 0$ (or strictly $\nu_{\alpha',1,H}$-tilt-semistable), then we consider the Jordan--H\"older filtration of $F'$. Since $f$ is star-shaped along the line $\beta = 0$ and $\beta = d$, there is at least one Jordan--H\"older factor $F'_i$ with $\mu_H(F'_i)\in(0,1)$ also satisfying $\xi_H(F'_i) > f(\mu_H(F'_i))$. By \Cref{RemarkOnDiscriminats}, we have $\overline{\Delta}_H(F') > 0$. Hence $\overline{\Delta}_H(F'_i) < \overline{\Delta}_H(F')$ by \Cref{LemmaDiscriminant}, and this violates the minimum assumption on $\overline{\Delta}_H(F')$. As a consequence, we may assume $F'$ is $\nu_{\alpha,0,H}$-tilt-stable (or $\nu_{\alpha',1,H}$-tilt-stable) for any $\alpha > 0$ (or $\alpha'>\frac{1}{2}$).

    If $F'$ becomes strictly $\nu_{\alpha,\beta_0,H}$-tilt-semistable at the vertical wall for $\beta_0 = \mu_H(F')$ and some $\alpha > \frac{\beta_0^2}{2}$, we may assume that $F'\in\Coh^{\beta_0,H}(S)$. Then each torsion Jordan--H\"older factor of $F'$ has $\ch{2}\geq 0$. Since for any other Jordan--H\"older factor $F'_j$ we have $0>\rk(F'_j) \geq \rk(F')$, there exists a factor $F'_i$ with $\mu_H('F_i) = \mu_H(F')$ and $\xi_H(F'_i) \geq \xi_H(F')$. In particular, the object $F'_i$ also satisfies $\xi_{H}(F'_i) > f(\mu_H(F'_i))$. and $\overline{\Delta}_H(F'_i)\leq \overline{\Delta}_H(F')$. By \Cref{StructureOfWalls}, $F'_i$ is tilt-stable in a neighborhood of $(\beta_0,\alpha)$, and hence $\nu_{\alpha,0,H}$-tilt-stable and $\nu_{\alpha,1,H}$-tilt-stable for $\alpha \gg 0$. Since $\overline{\Delta_H}(F')$ is minimal, $\overline{\Delta}_H(F') = \overline{\Delta}_H(F'_i)$, then by the previous argument, $F'_i$ is $\nu_{\alpha,0,H}$-stable for all $\alpha > 0 $ and $\nu_{\alpha',1,H}$-stable for all $\alpha' > \frac{1}{2}$. Take $F$ to be $F'_i$.  As a consequence, we may now assume $F$ is $\nu_{\alpha,0,H}$-tilt-stable for any $\alpha > 0$ and $\nu_{\alpha',d,H}$-tilt-stable for any $\alpha' > d^2/2$. Since $F$ is $\nu_{\alpha,0,H}$-tilt-stable for $\alpha \gg 0$, the object $F$ is a coherent sheaf. And the other statement follows.
\end{proof}

\subsection{Inequality on quintic surfaces}
We next deduce a stronger Bogomolov--Gieseker type inequality for quintic surfaces, using the Clifford-type bounds from \Cref{section3}.

\begin{figure}[H]
    \centering
    \begin{tikzpicture}[line cap=round,line join=round,x=10cm,y=10cm]
  \begin{axis}[
    axis lines=middle,
    xmin=0, xmax=1.2,
    ymin=-0.35, ymax=0.65,
    xlabel={$\mu_H$},
    ylabel = {$\xi_H$},
    ylabel style={font=\Large},
    xtick={13/20, 40/47, 1},
    xticklabels={$\frac{13}{20}$, $\frac{40}{47}$, $1$},
    extra x ticks={7/47, 7/20},
    extra x tick labels=\empty,
    ytick={-3/20,0.5},
    yticklabels={$-\frac{3}{20}$, $\frac{1}{2}$},
    width= 15cm, height = 12cm
  ]
    \addplot[red,thick,domain=0:1]{0.5*x*x};
    \node[font=\Large,red] at (axis cs:0.65,0.32) {$y=\frac{1}{2}x^2$};
    
    \addplot[thick,domain=0:7/47]{13/20*x-3/20};
    \addplot[thick,domain=7/47:7/20]{-5/14*x};
    \addplot[thick,domain=7/20:13/20]{0.5*x-0.3};
    \addplot[thick,domain=13/20:40/47]{19/14*x - 6/7};
    \addplot[thick,domain=40/47:1]{7/20*x};
    \node[anchor=south,font=\small] at (axis cs:7/46,0.01) {$\frac{7}{47}$};
    \node[anchor=south,font=\small] at (axis cs:7/20,0) {$\frac{7}{20}$};
    
    \draw[thick,fill=white] (axis cs:2,5) circle (1.5pt);
    \filldraw (axis cs:2,6) circle (1.5pt);

    \draw [line width=1pt,dash pattern=on 1pt off 5pt] (1,0)-- (1,0.5);
    \draw [line width=1pt,dash pattern=on 1pt off 5pt] (0,0.5)-- (1,0.5);
    \draw [line width=1pt,dash pattern=on 1pt off 5pt] (7/47,0)-- (7/47,-5/92);
    \draw [line width=1pt,dash pattern=on 1pt off 5pt] (7/20,0)-- (7/20,-1/8);
    \draw [line width=1pt,dash pattern=on 1pt off 3pt] (13/20,0)-- (13/20,1/40);
    \draw [line width=1pt,dash pattern=on 1pt off 5pt] (40/47,0)-- (40/47,27/92);
  \end{axis}
\end{tikzpicture}

    \caption{BG type inequality on quintic surfaces}
    \label{fig:BGsurface}
\end{figure}
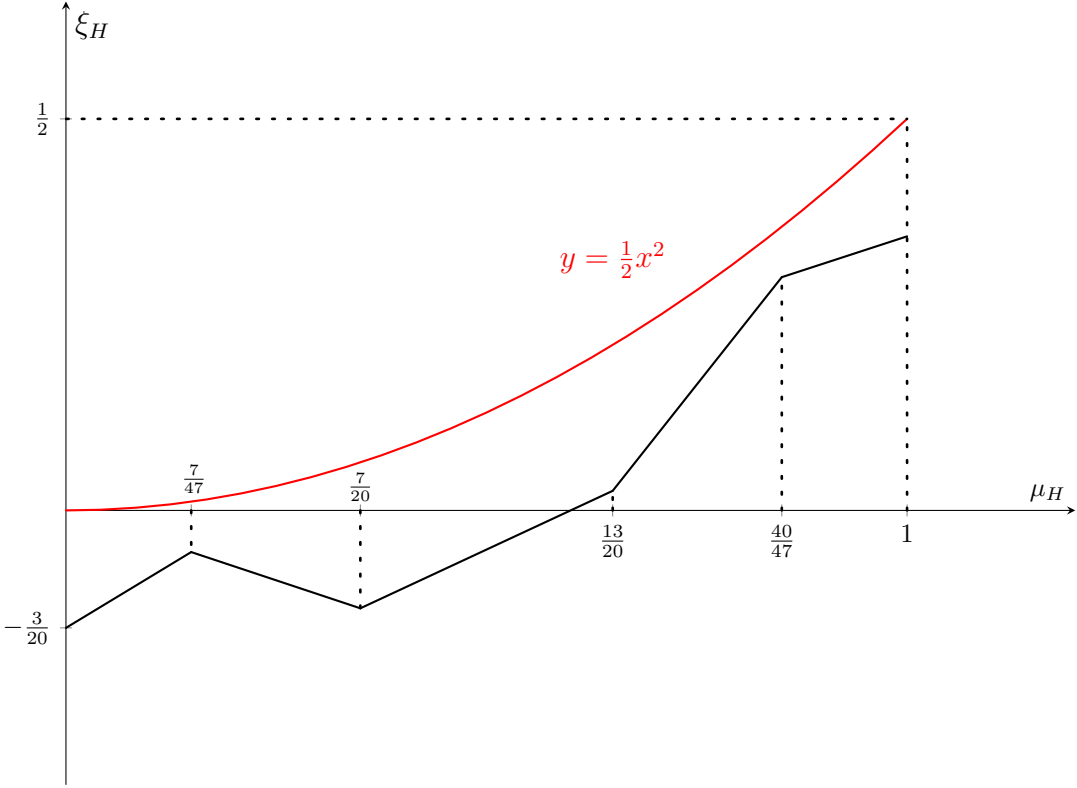

\begin{proposition}\label{ChInequalityOnSurface}
    Let $S_5\subset \mathbb{P}^3$ be a smooth irreducible quintic surface, $H = c_1(\mathcal{O}_{S_5}(1))$, and let $F$ be an object in $\Db(S_5)$ such that $\mu_H(F)\in (0,1)$. Suppose $F$ is $\nu_{\alpha,0,H}$-tilt-stable (or $\nu_{\alpha',1,H}$-tilt-stable) for some $\alpha > 0$ (resp. or $\alpha' > \frac{1}{2}$), then
    \begin{equation}\label{Ch1Ch2Inequal}
        \xi_H(F)\leq 
    \begin{cases}
    \frac{13}{20}\mu_H(F) - \frac{3}{20} & \text{when } 0 < \mu_H(F) \leq \frac{7}{47}; \\[0ex]
    -\frac{5}{14}\mu_H(F)  & \text{when } \frac{7}{47} \leq  \mu_H(F)\leq \frac{7}{20}; \\[0ex]
    \frac{1}{2}\mu_H(F)-\frac{3}{10} & \text{when } \frac{7}{20}\leq \mu_H(F) \leq \frac{13}{20}; \\[0ex]
    \frac{19}{14}\mu_H(F) - \frac{6}{7} & \text{when } \frac{13}{20}\leq \mu_H(F)< \frac{40}{47}; \\[0ex]
    \frac{7}{20}\mu_H(F) & \text{when } \frac{40}{47}\leq \mu_H(F)< 1.
\end{cases}
    \end{equation}
\end{proposition}

\begin{proof}
    We argue by contradiction. Suppose there is a $\nu_{\alpha,0,H}$-tilt-stable or $\nu_{\alpha', 1, H}$-tilt-stable object $F$ with $\mu_H(F)\in (0,1)$ violating the above inequality. Define $f_{S_5}: [0,1]\rightarrow \mathbb{R}$ as that on the right hand side of \eqref{Ch1Ch2Inequal}, viewing $\mu_H(F)$ as the variable: 
    \begin{equation*}
        f_{S_{5}}(x) =  
    \begin{cases}
    \frac{13}{20}x - \frac{3}{20} & \text{when } 0 < x \leq \frac{7}{47}; \\
    -\frac{5}{14}x  & \text{when } \frac{7}{47} \leq  x \leq \frac{7}{20}; \\
    \frac{1}{2}x-\frac{3}{10} & \text{when } \frac{7}{20}\leq x \leq \frac{13}{20}; \\
    \frac{19}{14}x - \frac{6}{7} & \text{when } \frac{13}{20}\leq x < \frac{40}{47}; \\
    \frac{7}{20}x & \text{when } \frac{40}{47}\leq x < 1.
\end{cases}
    \end{equation*}
    It is direct to check that $f_{S_5}$ and $F$ satisfy the assumption of \Cref{LemmaReductionToStableObjects} 
    for $d = 1$. By \Cref{LemmaReductionToStableObjects}, we may assume that $F$ is a coherent sheaf. Moreover, $F\in\Coh^{0,H}(S_5)$ is $\nu_{\alpha,0,H}$-stable for any $\alpha > 0$, and that $F[1] \in \Coh^{1,H}(X)$ is $\nu_{\alpha',1,H}$-tilt-stable for any $\alpha' > 1/2$. 

    Without loss of generality, we may assume $\mu_H(F) \in (0,\frac{1}{2}]$, since otherwise we might replace $F$ by $F^{\vee}(H)$. 
    
    By \Cref{RestrictionTheorem}, we might take $C\in|H|$, such that for any $\alpha > 0$, the restricted sheaf $F|_{C}$ has 
    \begin{align*}
        [\mu_{H_C}^-(F|_C),\mu^+_{H_C}(F|_C)]\subseteq \llbracket \tfrac{1}{2}+\nu_{\alpha,0,H}(F), \tfrac{1}{2}+\nu_{\alpha,0,H}(F(-H)[1]) \rrbracket.
    \end{align*} 
    Letting $\alpha\rightarrow 0$, this spells as
    \begin{align*}
        [\mu_{H_C}^-(F|_C),\mu^+_{H_C}(F|_C)]\subseteq \llbracket \tfrac{1}{2}+\nu_{0,0,H}(F), \tfrac{1}{2}+\nu_{0,0,H}(F(-H)[1]) \rrbracket.
    \end{align*}
    The same argument applied to $F^\vee(H)$, we have
    \begin{equation*}
        [\mu_{H_C}^-(F^\vee(H)|_C),\mu^+_{H_C}(F^\vee(H)|_C)]\subseteq \llbracket \tfrac{1}{2}+\nu_{0,0,H}(F^\vee(H)), \tfrac{1}{2}+\nu_{0,0,H}(F^\vee[1]) \rrbracket.
    \end{equation*}
    We firstly consider the range $\tfrac{7}{47}\leq\mu_H(F)\leq\tfrac{1}{2}$. Note that $\Hom(\mathcal{O}_{S_5}, F(-H)) = 0$ and $\Hom(\mathcal{O}_{S_5},F^\vee) = 0$, by \Cref{RRQuinticSurface}, we have
    \begin{align}
        &\ch{2}(F) - \tfrac{1}{2}H\ch{1}(F)+ H^2\ch{0}(F) = \chi(\mathcal{O}_{S_5}, F)\notag \\
        \leq & \hom(\mathcal{O}_{S_5},F) + \hom(\mathcal{O}_{S_5},F[2]) = \hom(\mathcal{O}_{S_5},F) + \hom(\mathcal{O}_{S_5},F^\vee(H)) \notag\\
        \leq &\hom(\mathcal{O}_{C_5}, F|_{C_5}) + \hom (\mathcal{O}_{C_5},F^\vee(H)|_{C_5}). \label{RREstimate}
    \end{align}
    
By our assumption on $F$ and \eqref{eqnu}, we have 
\begin{align*}
    \mu^+_{H_{C}}(F|_C) &\leq \tfrac{1}{2} + \max\{\nu_{0,0,H}(F), \nu_{0,0,H}(F(-H)[1]) \} \\
                        & = \max\left\{\frac{\xi_{H}(F)}{\mu_H(F)} + \frac{1}{2}, \frac{\xi_H(F)- \tfrac{1}{2}\mu_H(F)}{\mu_H(F) - 1}\right\} \leq \frac{6}{7},
\end{align*}
and
\begin{align*}
    \mu^+_{H_{C}}(F^\vee(H)|_C) &\leq \tfrac{1}{2} + \max\{\nu_{0,0,H}(F^\vee(H)), \nu_{0,0,H}(F^\vee[1])\} \\
                                & = \max\left\{\frac{\xi_H(F)- \frac{1}{2}}{1-\mu_H(F)} + \frac{3}{2}, \frac{1}{2} - \frac{\xi_{H}(F)}{\mu_H(F)} \right\} \leq \frac{6}{7}.
\end{align*}

Then by \Cref{H0Bound}, the last term in  \eqref{RREstimate} is
\begin{align*}
    \leq \rk(F)\cdot\left(1+ \tfrac{3}{2}\mu_H(F) + 1 + \tfrac{3}{2}(1 - \mu_H(F))\right).
\end{align*}
This in turn gives:
\begin{equation}
    \ch{2}(F) \leq \tfrac{1}{2} H\ch{1}(F) -\tfrac{3}{10}H^2\rk(F).\notag
\end{equation}
And this violates our assumption on $F$ when $\tfrac{7}{47}\leq\mu_H(F)\leq\tfrac{1}{2}$.

When $0 < \mu_H(F) \leq \tfrac{7}{47}$, the condition $\xi_H(F) > \tfrac{13}{20}\mu_H(F) - \tfrac{3}{20}$  implies that $\mu^+_{H_C}(F|_C) \leq \tfrac{3}{20}$, and  $\mu^+_{H_C}(F(H)|_C) \leq \frac{23}{20}$. Then by \Cref{H0Bound1}
\begin{align*}
        \chi(\mathcal{O}_{S_5}, F(H)) &= \ch{2}(F) + \tfrac{1}{2}H\ch{1}(F)+ H^2\ch{0}(F)  \\
        &\leq\hom(\mathcal{O}_{S_5},F(H)) + \hom(\mathcal{O}_{S_5},F(H)[2]) = \hom(\mathcal{O}_{S_5},F(H)) + \hom(\mathcal{O}_{S_5},F^\vee) \\
        &\leq \hom(\mathcal{O}_{C_5}, F(H)|_{C_5}) + \hom (\mathcal{O}_{C_5},F|_{C_5}) \\
        & \leq \rk(F)\cdot\left( \tfrac{7}{2}(\tfrac{3}{20} + 1) - 1  \right) + \rk(F)\cdot\left(\tfrac{3}{2}\cdot\tfrac{3}{20}+ 1\right) \\
        & = \tfrac{17}{20}H^2\ch{0}(F).
    \end{align*}
Hence $\ch{2}(F)\leq -\tfrac{1}{2}H\ch{1}(F)-\tfrac{3}{20}H^2\ch{0}(F)$. This contradicts our assumption since $\xi_H(F) > \tfrac{13}{20}\mu_H(F) - \tfrac{3}{20} > -\tfrac{1}{2}\mu_H(F)-\tfrac{3}{20}$ when $0 < \mu_H(F) \leq \tfrac{7}{47}$.
\end{proof}

\begin{corollary}\label{CorollaryOnSheaves}
    Let $F$ be a torsion-free $H$-slope-semistable sheaf on $S_5$, then the numerical Chern characters of $F$ satisfy \eqref{Ch1Ch2Inequal}.
\end{corollary}
\begin{proof}
    By considering Harder--Narasimhan factors, we may assume $F$ is $H$-slope-stable. Note that $F$ is $\nu_{\alpha,0,H}$-tilt stable for $\alpha\gg 0$ by \Cref{LargeVolumeLimit}. Then we may conclude by \Cref{ChInequalityOnSurface}.
\end{proof}


\subsection{Inequality on quintic threefolds}
Now we may prove the inequality on quintic threefolds.

\begin{lemma}\label{LemmaOnChOnThreefolds}
    Let $F$ be a torsion-free sheaf on $S_5$ with $\mu_H(F)\in(0,1)$. Assume that $\frac{7}{20}\leq \mu^-_H(F)\leq \mu^+_H(F) \leq \frac{13}{20}$, then
    \begin{equation*}
        \xi_H(F) \leq \frac{1}{2}\mu_H(F) - \frac{3}{10}.
    \end{equation*}
\end{lemma}
\begin{proof}
    The proof is similar to \Cref{H0Bound}. Just note that $\xi_H$ is additive and use \Cref{CorollaryOnSheaves}.
\end{proof}

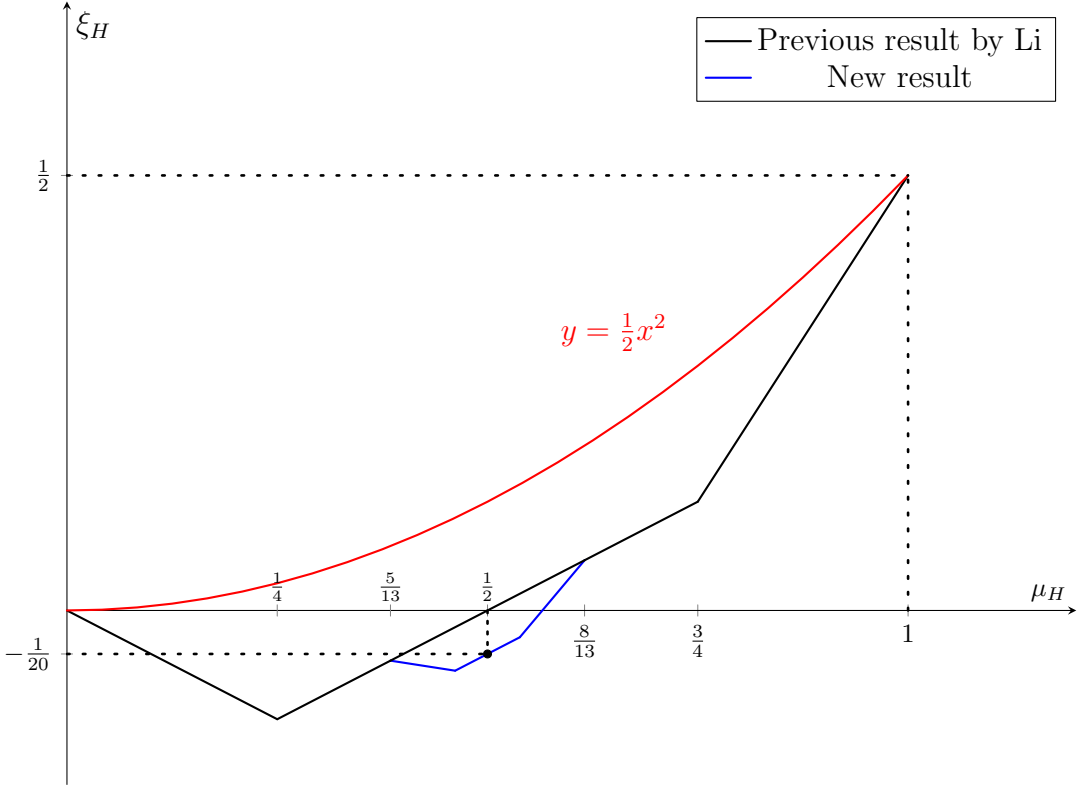
\begin{figure}[H]
    \centering
    \begin{tikzpicture}[line cap=round,line join=round,x=10cm,y=10cm]
  \begin{axis}[
    axis lines=middle,
    xmin=0, xmax=1.2,
    ymin=-0.2, ymax=0.7,
    xlabel={$\mu_H$},
    ylabel = {$\xi_H$},
    ylabel style={font=\Large},
    xtick={8/13, 3/4, 1},
    xticklabels={$\frac{8}{13}$, $\frac{3}{4}$, $1$},
    extra x ticks={1/4, 5/13, 1/2},
    extra x tick labels=\empty,
    ytick={-0.05,0.5},
    yticklabels={$-\frac{1}{20}$, $\frac{1}{2}$},
    width= 15cm, height = 12cm,
    legend style={font=\Large, at={(0.98,0.98)}, anchor=north east}  
  ]
    
    \addplot[thick,domain=0:1/4]{-0.5*x};
    \addplot[blue,thick,domain=5/13:6/13]{-3/20*x};
    \addlegendentry{Previous result by Li};
    \addplot[thick,domain=1/4:5/13]{0.5*x-0.25};

    \addplot[blue,thick,domain=6/13:7/13]{0.5*x - 0.3};
    \addplot[blue,thick,domain=7/13:8/13]{23/20*x-13/20};
    \addlegendentry{New result};
    \addplot[thick,domain=8/13:3/4]{0.5*x-0.25};
    \addplot[thick,domain=3/4:1]{1.5*x-1};
    \addplot[thick,domain=5/13:8/13]{0.5*x-0.25};
    \addplot[red,thick,domain=0:1]{0.5*x*x};
    \node[font=\Large,red] at (axis cs:0.65,0.32) {$y=\frac{1}{2}x^2$};
    
    \node[anchor=south,font=\small] at (axis cs:1/4,0) {$\frac{1}{4}$};
    \node[anchor=south,font=\small] at (axis cs:5/13,0) {$\frac{5}{13}$};
    \node[anchor=south,font=\small] at (axis cs:1/2,0) {$\frac{1}{2}$};
    
    \draw[thick,fill=white] (axis cs:2,5) circle (1.5pt);
    \filldraw (axis cs:2,6) circle (1.5pt);

    \draw [line width=1pt,dash pattern=on 1pt off 5pt] (1,0)-- (1,0.5);
    \draw [line width=1pt,dash pattern=on 1pt off 5pt] (0,0.5)-- (1,0.5);
    \draw [line width=1pt,dash pattern=on 1pt off 3pt] (1/2,0)-- (1/2,-0.05);
    \draw [line width=1pt,dash pattern=on 1pt off 5pt] (0,-0.05)-- (1/2,-0.05);

    \filldraw (axis cs:1/2,-0.05) circle (1.5pt);
  \end{axis}
\end{tikzpicture}

    \caption{BG type inequality on quintic threefolds}
    \label{fig:BGthreefold}
\end{figure}

\begin{theorem}\label{MainTheorem}
     Let $X\subset \mathbb{P}^4_{\mathbb{C}}$ be a smooth quintic $3$-fold and $H: = c_1(\mathcal{O}_X(1))$, $F$ be a $\nu_{\alpha,0,H}$-tilt-stable object in $\Coh^{0,H}(X)$ for some $\alpha >0$. Suppose $\mu_H(F)\in [-1,1]$. then
    \begin{equation}\label{InequalityOnThreefolds}
        \xi_H(F)\leq 
    \begin{cases}
    \frac{-1}{2}|\mu_H(F)| & \text{when } 0\leq |\mu_H(F)| \leq \frac{1}{4}; \\[0ex]
    \frac{1}{2}|\mu_H(F)| - \frac{1}{4} & \text{when } |\mu_H(F)|\in[\frac{1}{4},\frac{5}{13}]\cup[\frac{8}{13},\frac{3}{4}]; \\[0ex]
    \frac{-3}{20}|\mu_H(F)|  & \text{when } \frac{5}{13} \leq |\mu_H(F)|\leq \frac{6}{13}; \\[0ex]
    \frac{1}{2}|\mu_H(F)|-\frac{3}{10} & \text{when } \frac{6}{13}\leq |\mu_H(F)| \leq \frac{7}{13}; \\[0ex]
    \frac{23}{20}|\mu_H(F)| - \frac{13}{20} & \text{when } \frac{7}{13}\leq |\mu_H(F)| \leq \frac{8}{13}; \\[0ex]
    \frac{3}{2}|\mu_H(F)| - 1 & \text{when } \frac{3}{4}\leq |\mu_H(F)| \leq 1.
\end{cases}
    \end{equation}
\end{theorem}
\begin{proof}
    We argue by contradiction. Without loss of generality, we deal with the case $\mu_H(F)\in (0,1)$. As in \Cref{ChInequalityOnSurface}, we may assume $F$ is everywhere tilt-stable, violating \eqref{InequalityOnThreefolds}, and $F\in\Coh^{0,H}(X)$, and $F(-H)[1]\in\Coh^{0,H}(X)$. The inequality in the range $[0,\frac{5}{13}]\cup[\frac{8}{13},1]$ is just \cite[Theorem 5.5]{li_stability_2019}, and we only need to deal with the case $\mu_H(F)\in [\frac{5}{13},\frac{8}{13}]$. In this range, by our assumption, we have
    \begin{equation*}
        \nu_{0,0,H}(F)\leq \nu_{0,0,H}(F(-H)[1]).
    \end{equation*}
    
    By \Cref{RestrictionTheorem}, for any $\alpha > 0$, the restricted sheaf $F_{S_5}$ satisfies
    \begin{equation*}
        \left[\mu_{H_{S}}^-(F|_{S}),\mu^+_{H_S}(F|_S)\right] \subseteq \left[\frac{1}{2}+\nu_{\alpha,0,H}(F),\frac{1}{2}+ \nu_{\alpha,0,H}(F(-H)[1])\right].
    \end{equation*}
    Letting $\alpha\rightarrow 0$, by our assumption, this spells as
    \begin{equation*}
        \mu^-_{H_{S}}(F|_S) \geq \frac{1}{2}+\nu_{0,0,H}(F) = \frac{\xi_H(F)}{\mu_H(F)} + \frac{1}{2} > \frac{7}{20},
    \end{equation*}
    \begin{equation*}
        \mu^+_{H_{S}}(F|_S) \leq \frac{1}{2}+ \nu_{0,0,H}(F(-H)[1]) = \frac{\xi_H(F) - \frac{1}{2}\mu_H(F)}{\mu_H(F) - 1} < \frac{13}{20}.
    \end{equation*}
    By \Cref{LemmaOnChOnThreefolds}, we have
    \begin{equation*}
        \xi_H(F) = \xi_{H_S}(F|_S) \leq \tfrac{1}{2}\mu_{H_S}(F|_S) - \tfrac{3}{10} =\tfrac{1}{2}\mu_H(F) - \tfrac{3}{10}. 
    \end{equation*}
    This violates our assumption.
\end{proof}

\begin{remark}
    Compared to the result in \cite{li_stability_2019}, our approach yields a stronger bound in the neighborhood of $\frac{1}{2}$. However, near $0$ and $1$, the behavior of the maximal(minimal) slope becomes more intricate, and our method appears to be less effective in those regions.
\end{remark}

\begin{corollary}\label{MainCorollary}
    Let $F$ be a torsion-free $H$-slope-semistable sheaf on $X$, then the numerical Chern characters of $F$ satisfy \eqref{InequalityOnThreefolds}.
\end{corollary}
\begin{proof}
    Same as \Cref{CorollaryOnSheaves}.
\end{proof}

\begin{remark}
    The argument above refines the method developed in \cite{li_stability_2019} and \cite{Naoki23_BG_hypersurfaces} by employing the restriction to $|H|$ instead of $|2H|$, together with explicit Clifford-type bounds for plane quintic curves. This improvement leads to sharper bound near $\frac{1}{2}$.

    It remains an open question whether these inequalities suffice to construct a Bridgeland stability condition of Gepner type on the quintic threefold. Indeed, the $(\beta,\alpha)$ in the family of stability conditions $\{\sigma^{a,b}_{\alpha,\beta,H}\}$ constructed in \cite{li_stability_2019} needs to satisfy
    \begin{equation*}
        \alpha^2 + (\beta - \lfloor\beta\rfloor  - \tfrac{1}{2} )^2 > \tfrac{1}{4},
    \end{equation*}
    which is even above the parabola $\alpha = \frac{1}{2}\beta^2$.

    We hope to address these issues in future work.
\end{remark}

\bibliography{gepner}
\bibliographystyle{alpha}

\end{document}